\documentclass[12pt]{amsart}
\usepackage{SSdefn}
\usepackage[nobysame,backrefs]{amsrefs}

\newcommand{\HF}{\mathrm{HF}}

\title[Uniformity for bounded degree polynomials]{Cubics in 10 variables vs. cubics in 1000 variables:\\
 Uniformity phenomena for bounded degree polynomials}

\author{Daniel Erman}
\address{Department of Mathematics, University of Wisconsin, Madison, WI}
\email{\href{mailto:derman@math.wisc.edu}{derman@math.wisc.edu}}
\urladdr{\url{http://math.wisc.edu/~derman/}}

\author{Steven V Sam}
\address{Department of Mathematics, University of California, San Diego, CA}
\email{\href{mailto:ssam@ucsd.edu}{ssam@ucsd.edu}}
\urladdr{\url{http://math.ucsd.edu/~ssam/}}

\author{Andrew Snowden}
\address{Department of Mathematics, University of Michigan, Ann Arbor, MI}
\email{\href{mailto:asnowden@umich.edu}{asnowden@umich.edu}}
\urladdr{\url{http://www-personal.umich.edu/~asnowden/}}

\thanks{DE was partially supported by NSF DMS-1601619.\\
  \indent SS was partially supported by NSF DMS-1500069, NSF DMS-1812462, and a Sloan Fellowship.\\
\indent  AS was partially supported by NSF DMS-1453893.}

\date{\today}

\begin{document}

\maketitle

\section{Introduction}
In two landmark papers~\cite{hilbert1890,hilbert1893}, Hilbert laid the foundations for the modern algebraic study of polynomials. The theorems at the heart of these papers show that polynomials in $n$ variables are not too complicated, in various senses. For example, the Hilbert Syzygy Theorem shows that the process of resolving a module by free modules terminates in finitely many (in fact, at most $n$) steps, while the Hilbert Basis Theorem shows that the process of finding generators for an ideal also terminates in finitely many steps. Hilbert used his theorems to show that invariant rings are finitely generated, resolving one of the central problems of his day; in the years since, entire fields of mathematics have been built around Hilbert's results.

Obviously, polynomials in $n$ variables will typically exhibit greater complexity as $n$ increases. In other words, Hilbert's theorems are not uniform in $n$. However, an array of recent work has shown that in a certain regime---namely, that where the number of polynomials and their degrees are fixed---the complexity of polynomials remains bounded, at least according to a wide variety of measures. We refer to this phenomenon as {\bf Stillman uniformity}, as Stillman's Conjecture is the model case. The purpose of this paper is to give an exposition of Stillman uniformity and some of the work around it.

Our account is focused on four closely related threads of work, which we now introduce and briefly summarize.

\medskip

\noindent {\bf I. Stillman's Conjecture.} The first indication, as far as we are aware, of the general phenomenon of Stillman uniformity can be found in a conjecture posed by Michael Stillman around the year 2000\footnote{It first appeared in print in \cite[\S1]{engheta2}.  See also~\cite[Problem 3.14]{PS}.}. Recall that the {\bf projective dimension} of a module is the minimal length of a projective resolution (see \S \ref{s:stillman} for a review). This is a fundamental, albeit rather technical, invariant. Hilbert's Syzygy Theorem is exactly the statement that every module over an $n$-variable polynomial ring has projective dimension at most $n$. Stillman's Conjecture refines this theorem: it asserts that the projective dimension of an ideal in an $n$-variable polynomial ring generated by $r$ homogeneous polynomials\footnote{A polynomial is {\bf homogeneous} if all terms have the same degree.  For instance $x_1^3+x_1x_2x_3$ is homogeneous of degree $3$, but $x_1^3+x_1^2$ is not homogeneous.} of degrees $\le d$ can be bounded in terms of $r$ and $d$, but independent of the number $n$ of variables. In other words, in this particular regime, Hilbert's Syzygy Theorem holds uniformly in $n$. Stillman's Conjecture was proved by Ananyan and Hochster \cite{ananyan-hochster} in 2016, and has subsequently been reproven by us \cite{stillman} and Draisma, Laso\'n, and Leykin \cite{draisma-lason-leykin}.

\medskip

\noindent {\bf II. The Ananyan--Hochster principle.} In the course of proving Stillman's Conjecture, Ananyan and Hochster prove a number of fundamental results on the structure of polynomials. These results can be seen as special cases of the following general principle, which we call the Ananyan--Hochster principle: given homogeneous polynomials $f_1, \ldots, f_r$ of degrees $\le d$ in any number $n$ of variables, one can write $f_i=F_i(g_1, \ldots, g_s)$ where $F_i$ is a polynomial and $g_1, \ldots, g_s$ are homogeneous polynomials of degrees $\le d$ such that (a) $s$ depends only on $d$ and $r$ (and, crucially, not on $n$); and (b) $g_1, \ldots, g_s$ behave approximately like $s$ independent variables. In other words, the $f_i$'s look like polynomials in $s$ variables. For instance, four cubic polynomials in $1000$ variables (or in $10^{10}$ or $10^{100}$ variables) will behave like polynomials in $s$ variables for some fixed $s$.  One can therefore expect the $f_i$'s to satisfy the same sort of finiteness properties that Hilbert established, and thereby obtain Stillman uniformity.

Of course, it is crucial here to understand the exact meaning of condition (b). In fact, there are many possible precise meanings, and each yields a definite statement that may or may not be true. 
Ananyan and Hochster proved a number of incarnations of the principle, and subsequently other incarnations have been proved as well.  This is discussed in much more detail in \S\ref{s:principle}--\S\ref{s:reggerm}.

\medskip

\noindent {\bf III. Big polynomial rings.} 
We say that a homogeneous polynomial $f$ is {\bf $n$-decomposable} if it can be written in the form $f=F(g_1, \ldots, g_n)$, where $F$ is a polynomial and the $g_i$ are homogeneous polynomials of smaller degree than $f$. This is one way measure of the complexity of $f$. For homogeneous polynomials $f_1, \ldots, f_r$, we let $\nu(f_1, \ldots, f_r)$ be the minimum value of $n$ such that some non-trivial homogeneous linear combination of the $f_i$'s is $n$-decomposable. This is a kind of measure of the joint complexity of the $f_i$'s. The Ananyan--Hochster principle easily reduces to the claim: if $\nu(g_1, \ldots, g_s)$ is large then $g_1, \ldots, g_s$ behave approximately like independent variables. This is an asymptotic statement: as $\nu$-complexity increases, so does the approximation to independent variables.

A general principle of mathematics is that it is often useful to take the limit of an asymptotic statement to obtain an exact statement. Indeed, the limiting statement is often cleaner and can reveal deeper truths about the asymptotic situation. In \cite{stillman}, we applied this philosophy to the Ananyan--Hochster principle. We defined two rings $\bR$ and $\bS$, which, in this paper, we refer to as the ring of bounded-degree series and the ring of bounded-degree germs. These rings can be viewed as two different limits\footnote{As we will see, the ring $\bR$ arises as an inverse limit, a common algebraic construction.  The ring $\bS$ is a bit more exotic, and involves the model-theoretic notion of an ultraproduct.} of the $n$-variable polynomial rings as $n$ tends to infinity. One can define the quantity $\nu(g_1, \ldots, g_s)$ for $g_1, \ldots, g_s$ in either of these rings; in contrast to the polynomial case, this invariant is often infinite in these rings. 

Any element $f$ of $\bR$ or $\bS$ can be expressed in the form $F(g_1, \ldots, g_s)$ where $F$ is a polynomial and $\nu(g_1, \ldots, g_s)$ is infinite. The 
Ananyan--Hochster principle suggests that if $g_1, \ldots, g_s$ have infinite $\nu$-complexity, then they should behave {\em exactly} like $s$ independent variables.
One of the main theorems of \cite{stillman} verifies this: if $g_1, \ldots, g_s$ have infinite $\nu$-complexity, then they literally are independent variables, up to an isomorphism; that is, the rings $\bR$ and $\bS$ are abstractly isomorphic to polynomial rings (in uncountably many variables; hence ``big''), and under the isomorphism, $g_1, \ldots, g_s$ correspond to distinct variables. These theorems provide idealized forms of the Ananyan--Hochster principle.

These idealized forms are not only aesthetically pleasing statements, they are also useful: one can deduce some of the most important incarnations of the  Ananyan--Hochster principle from them. In fact, the cleanest proof of Stillman's Conjecture, in our opinion, proceeds by first proving the idealized Ananyan--Hochster principle for $\bS$, then deducing an instance of the ordinary Ananyan--Hochster principle from this, and finally deducing Stillman's Conjecture from this. We explain this line of reasoning in the body of the paper.

\medskip

\noindent {\bf IV. $\GL$-noetherianity.} Let $X_d$ be the space of homogeneous polynomials of degree $d$ in variables $x_1, x_2, \ldots$. A homogeneous polynomial of degree $d$ can be written in the form $\sum_{\alpha} c_{\alpha} x^{\alpha}$, where the sum is over multi-indices $\alpha$ of degree $d$ and the $c_{\alpha}$ are complex numbers, all but finitely many of which vanish. We thus see that $X_d$ can be identified with the space of tuples $(c_{\alpha})$, and that $X_d$ is therefore isomorphic to an infinite dimensional complex vector space. The group $\GL_{\infty}$ acts on $X_d$ via linear substitutions in the variables.

Draisma \cite{draisma} proved the following fundamental finiteness result: $X_d$ is a $\GL_{\infty}$-noetherian variety. (More generally, he showed that $X_{d_1} \times \cdots \times X_{d_r}$ is $\GL_{\infty}$-noetherian, for any $d_1, \ldots, d_r$.) 
The precise meaning of this theorem, and the way in which it extends the Hilbert Basis Theorem, is spelled out in \S\ref{s:GLnoeth}.  

Draisma's Theorem is closely related to the Stillman uniformity phenomenon. In \cite{stillman}, we combined Draisma's Theorem and our idealized Ananyan--Hochster principle for $\bR$ to give an entirely different, and more geometric proof, of Stillman's Conjecture.  In \cite{draisma-lason-leykin}, Draisma, Laso\'n, and Leykin gave yet another proof of Stillman's Conjecture, deducing it from Draisma's Theorem and establishing some additional important finiteness results. In \cite{genstillman}, we combined Stillman's Conjecture and Draisma's Theorem to prove a vast generalization of Stillman's Conjecture, where the invariant ``projective dimension'' is replaced by an arbitrary invariant satisfying a few simple axioms. This indicates that Stillman uniformity really is a far reaching phenomenon.

\begin{remark}
Throughout, we will assume that the coefficients of all polynomials are complex numbers.  
It is possible to allow other fields but the discussion becomes more subtle.
Thus, for example, restricting to polynomials with real coefficients would change the discussion somewhat.  See~\S\ref{subsec:field} for a discussion of how the results depend on the field of coefficients.  
\end{remark}

\medskip

This paper is organized as follows:
\begin{itemize}
\item In \S \ref{s:homoseq}--\S \ref{s:principle} we state the Ananyan--Hochster principle in general, as well as several precise incarnations of it. The aim here is to state rigorous and interesting results that require minimal background to understand. No indication of proofs is given.
\item In \S \ref{s:series}--\S \ref{s:germs}, we introduce the rings $\bR$ and $\bS$ and state the idealized forms of the Ananyan--Hochster principle. The main aim is to motivate the introduction of these rings and explain the idealized principle; however, we also give a brief account of the proof, which is entirely elementary.
\item  In \S \ref{s:reggerm}--\S \ref{s:proof of stillman}, we explain the connections between Stillman's Conjecture, the Ananyan--Hochster principle, and the idealized Ananyan--Hochster principle. More precisely, in \S \ref{s:reggerm} we explain how to deduce the most important incarnation of the Ananyan--Hochster principle from the idealized principle for $\bS$. In \S \ref{s:stillman}, we review syzygies in general and precisely formulate Stillman's Conjecture. In \S \ref{s:proof of stillman}, we explain how to deduce Stillman's Conjecture from the incarnation of the Ananyan--Hochster principle established in \S \ref{s:reggerm}. Thus, by the end of \S \ref{s:proof of stillman}, we will have explained all of the key steps in one of the proofs of Stillman's Conjecture.
\item In \S\ref{s:GLnoeth}, we discuss Draisma's Theorem and its connections to Stillman uniformity.
\item Finally, in \S \ref{s:topics}, we briefly review an array of related results and further topics.
\end{itemize}
We hope that any reader with a general mathematical background should be able to follow the material up to \S \ref{s:germs} without great difficulty. The material in \S \ref{s:reggerm}--\S \ref{s:proof of stillman} is more specialized, but we have attempted to make it self-contained. The material in the final two sections relies on more background and is not self-contained, though we have tried to make it as accessible as possible.

\subsection*{Acknowledgments}
We would like to thank Craig Huneke, as his talk at the 2018 JMM Current Events Bulletin about Ananyan and Hochster's work was very influential in this paper.  We also thank Jan Draisma, Mel Hochster, and Jason McCullough for many thought-provoking discussions related to these topics.

\section{Decomposing polynomials} \label{s:homoseq}

There are many sensible ways that one could attempt to measure the complexity of polynomials. For the purposes of this paper, we consider a polynomial to be ``simple'' if it can be decomposed into a small number of lower degree polynomials. To this end, we recall the following definition from the introduction:

\begin{definition}
A homogeneous polynomial $f$ is {\bf $n$-decomposable} if there exist homogeneous polynomials $g_1, \ldots, g_n$ of strictly lower degree and a polynomial $F(X_1, \ldots, X_n)$ such that $f=F(g_1, \ldots, g_n)$. We let the {\bf $\nu$-complexity} of $f$, denoted $\nu(f)$, be the minimal $n$ for which $f$ is $n$-decomposable, with the convention that $\nu(f)=0$ if $f$ is constant and $\nu(f) = \infty$ if $f$ is a non-zero linear form.
\end{definition}

\begin{example} \label{ex:cubic}
If $f=(x_1^2+x_2^2+x_3^2)^3$ then choosing $g_1=x_1^2+x_2^2+x_3^2$ and $F=X_1^3$ shows that $\nu(f)\leq 1$.  But $\nu(f)$ cannot equal $0$ unless $f$ is constant, and thus $\nu(f)=1$.
\end{example}

\begin{example}\label{ex:quadric}
Suppose that $f=x_1^2+\cdots+x_n^2$. We claim that $\nu(f)=n$, i.e., that $f$ is not $(n-1)$-decomposable. To see this, suppose that $f=F(g_1, \ldots, g_m)$ with $m<n$ and each $g_i$ of degree $<2$. The $g_j$'s of degree~0 are simply scalars, and can be absorbed into $F$; thus we can assume each $g_i$ has degree~1. Thus $F(X_1, \ldots, X_m)$ is itself a homogeneous polynomial of degree~2. The expression $f=F(g_1, \ldots, g_m)$ shows that $f$ has rank $\le m$, in the sense of quadratic forms. However, we know that $f$ has rank $n$, which is a contradiction.
\end{example}

\begin{example}\label{ex:n-decomposable}
If $f$ has degree $\geq 2$ and uses the variables $x_1, \ldots, x_n$ then $f$ is necessarily $n$-decomposable, as one can take $g_i=x_i$ and $F=f$.  Hence $\nu(f) \le n$.  However, it can be the case that the $\nu$-complexity of $f$ is much smaller than the number of variables needed to express $f$.  For instance, if $f=(x_1^2+x_2^2+\cdots + x_n^2)^3$, then $f$ requires $n$ variables and yet, as in Example~\ref{ex:cubic}, one can check that $\nu(f)=1$.
\end{example}

A collection of homogeneous polynomials $\{f_1,\dots,f_r\}$ is {\bf $n$-decomposable}, if some non-trivial 
homogeneous linear combination $\alpha_1 f_1 + \cdots + \alpha_r f_r$ is $n$-decomposable (the $\alpha_i$ are complex numbers, not all $0$). As above, we define $\nu(f_1,\dots,f_r)$ to be the minimal $n$ such that $\{f_1,\dots,f_r\}$ is $n$-decomposable.

We note two extreme cases. If the $f_i$ are linearly dependent, then $\nu(f_1,\dots,f_r)= 0$. On the other hand, $\nu(f_1,\dots,f_r) = \infty$ if and only if all $f_i$ are linearly independent forms of degree one.

\begin{remark}\label{rmk:strength}
Ananyan--Hochster \cite{ananyan-hochster} define a homogeneous polynomial $f$ to have {\bf strength} $k$ if there is an expression of the form $f = \sum_{i=0}^k g_i h_i$ where the $g_i$ and $h_i$ are homogeneous polynomials of strictly smaller degree than $f$, and $k$ is minimal as such. Strength and $\nu$-complexity are asymptotically equivalent, in the sense that one is large if and only if the other is.
\end{remark}

\section{The principle of Ananyan--Hochster} \label{s:principle}

In their proof of Stillman's Conjecture \cite{ananyan-hochster}, Ananyan and Hochster discovered the following principle that predicts the behavior of polynomials that have large $\nu$-complexity:

\vskip.5\baselineskip\noindent
{\bf Ananyan--Hochster principle.} If $f_1, \ldots, f_r$ are homogeneous polynomials such that the $\nu$-complexity $\nu(f_1, \ldots, f_r)$ is large compared to $r$ and to the degrees of the $f_i$'s, then $\{f_1, \ldots, f_r\}$ behaves approximately like a set of $r$ independent variables.
\vskip.5\baselineskip\noindent

This is just a principle, not a theorem, since the statement is imprecise. Ananyan and Hochster proved several precise theorems that motivated the statement of this principle, and subsequently more instances of this principle were discovered. In the rest of this section, we look at some of the incarnations of the principle.

The principle can be used to generate precise predictions as follows. Start with a general algebraic property that holds for independent variables. The principle then predicts that if $f_1,\dots,f_r$ are homogeneous polynomials with $\nu(f_1,\dots,f_r)$ sufficiently large (relative to the degrees of $f_1,\dots,f_r$), then $f_1,\dots,f_r$ will also satisfy this property. For example, independent variables are {\bf algebraically independent} in the sense that there are no non-trivial polynomial relations among them. The principle thus predicts that if $f_1, \ldots, f_r$ have sufficiently high $\nu$-complexity then they too should be algebraically independent. In fact, this prediction is correct, and is implied by one of the results proven in \cite{ananyan-hochster}.

Here is a deeper consequence.  A basic fact from linear algebra is that the solution set to $r$ linearly independent linear equations in $x_1,\dots,x_n$ is a subspace of codimension $r$, i.e., is a subspace of dimension $n-r$. The natural generalization of this property to higher degree polynomials often fails. For example, consider the solution set to the equations $xy=xz=0$ in variables $x,y,z$. The solution set contains the subspace $x=0$ and hence has (complex) codimension~$1$, rather than the expected codimension $2$, even though $xy$ and $xz$ are linearly (and even algebraically) independent. We say that homogeneous polynomials $f_1,\dots,f_r$ in variables $x_1,\dots,x_n$ are a {\bf regular sequence} when the locus in $\bC^n$ defined by $f_1=\cdots=f_r=0$ has codimension $r$.
The polynomials in a regular sequence are automatically algebraically independent, and for many applications in commutative algebra and algebraic geometry, this is the most useful notion of independence.

The solution set $x_1=\cdots = x_r=0$ of $r$ independent variables always has codimension $r$.  Thus the Ananyan--Hochster principle suggests the following theorem, which was first proven by Ananyan--Hochster.  It remains one of the most important instances of the general principle, and we will return to it \S\ref{s:reggerm}.

\begin{theorem} \label{thm:reg}
If $f_1, \ldots, f_r$ are homogeneous polynomials of degrees at most $d$ such that $\nu(f_1, \ldots, f_r) \gg d,r$ then $f_1, \ldots, f_r$ is a regular sequence.
\end{theorem}

The above theorem essentially says that the Ananyan--Hochster principle holds for codimension. Here are some other properties for which it holds. In the following list, we assume that $f_1, \ldots, f_r$ are polynomials (in some unspecified number of variables $n$) with $\deg(f_i)\leq d$ for all $1\leq i \leq r$, and that $\nu(f_1, \ldots, f_r) \gg d,r$.
\begin{itemize}
\item Irreducibility and connectedness: the solution set $f_1=\cdots=f_r=0$ is irreducible. (Recall that an algebraic set is {\bf irreducible} if it is not the union of two proper closed algebraic sets. The typical example of a reducible algebraic set is the solution set of $xy=0$, which is the union of the loci defined by $x=0$ and $y=0$.)  In particular, this set is also connected.
\item Primality: the ideal of $\bC[x_1, \ldots, x_n]$ generated by $f_1, \ldots, f_r$ is prime. This is slightly stronger than irreducibility.
\item Smoothness: the set of singular points of the solution set of $f_1=\cdots=f_r=0$ has large codimension. Precisely, given any $c$ the singular locus has codimension $\ge c$ assuming $\nu(f_1,\ldots,f_r) \gg c,d,r$.
\item Cohomology: the solution set of $x_1=\cdots=x_r=0$ is isomorphic to the affine space $\bC^{n-r}$. The compactly supported cohomology of this space is easy to compute: the top group is $\bZ$, and all other groups vanish. The same is true for the algebraic set $X$ defined by $f_1=\cdots=f_r=0$, in the following sense. Given any $k$, the top $k$ compactly supported cohomology groups of $X$ agree with those of $\bC^{n-r}$, assuming $\nu(f_1, \ldots, f_r) \gg k,d,r$.
\end{itemize}
The first three examples in this list follow from~\cite[Theorem~A]{ananyan-hochster}. The one about cohomology follows from the one about smoothness and a result of Dimca; see \cite{kazhdan-schlank}.

\begin{remark}
One cannot expect the Ananyan--Hochster principle to apply to {\em every} property of independent variables.  For instance, independent variables define a solution set $x_1=\cdots=x_r=0$ where every point is smooth; by contrast, for homogeneous polynomials $f_1,\dots,f_r$ of degree $>1$, the solution set will always be singular at the origin of $\bC^n$.  It remains an interesting open problem to determine exactly which incarnations of the Ananyan--Hochster principle are true.  One recent result in this direction is provided by Bik, Draisma, and Eggermont~\cite{bik-draisma-eggermont}, as we discuss in \S\ref{subsec:bde}. 
\end{remark}

\section{Homogeneous series} \label{s:series}

The Ananyan--Hochster principle, as we have formulated it, is an asymptotic statement: as $\nu(f_1, \ldots, f_r)$ grows, the polynomials $f_1, \ldots, f_r$ more closely resemble $r$ independent variables. General mathematical principles suggest that we should try to construct a limiting situation where this approximation becomes exact. In this section, we exhibit one such limiting situation; \S\ref{s:germs} will exhibit another.

\subsection{Homogeneous series}
Recall, from Example~\ref{ex:quadric}, that $\nu(x_1^2+x_2^2+\cdots +x_n^2)=n$.   Taking the limit as $n$ tends to infinity suggests that the formal infinite sum $\sum_{i \ge 1} x_i^2$ should be indecomposable (i.e., not $n$-decomposable for any $n$), assuming that we can rigorously make sense of this statement. We now do just that.

A {\bf homogeneous series of degree $d$} is a formal sum $f=\sum_\alpha c_{\alpha} x^{\alpha}$ where the sum is over all multi-indices $\alpha$ of degree $d$ and the $c_{\alpha}$ are arbitrary complex numbers. (A multi-index is a sequence $\alpha=(\alpha_1, \alpha_2, \ldots)$ of non-negative integers such that all but finitely many are zero. The degree of a multi-index is $\alpha_1+\alpha_2+\cdots$ and $x^{\alpha}$ denotes the monomial $x_1^{\alpha_1} x_2^{\alpha_2} \cdots$.) For example, $\sum_{i \ge 1} x_i^2$ is a homogeneous series of degree~2. 

\begin{definition}
A homogeneous series $f$ is {\bf $n$-decomposable} if there exist homogeneous series $g_1, \ldots, g_n$ of strictly lesser degree and a polynomial $F(X_1, \ldots, X_n)$ such that $f=F(g_1, \ldots, g_n)$.   
We say that $f$ is {\bf indecomposable} if it fails to be $n$-decomposable for all $n$.

We say that homogeneous series $f_1, \ldots, f_r$ are {\bf jointly indecomposable} if every non-trivial homogeneous linear combination of them is indecomposable. An infinite collection of homogeneous series is  {\bf jointly indecomposable} if every finite subcollection is.
\end{definition}
To ensure that this definition makes sense, one needs to check that if $g_1,\dots,g_n$ are homogeneous series, and if $F(X_1,\dots,X_n)$ is appropriately homogeneous, then $F(g_1,\dots,g_n)$ is also a homogeneous series.  This amounts to checking two facts: first, if $f$ and $g$ are homogeneous series then so is $fg$, under the standard product of series; second, any linear combination of homogeneous series of equal degree is again a homogeneous series.  Both are easily verified.

We can also extend the definition of $\nu$ to homogeneous series in the obvious way. However, in this setting we will only ever care about $\nu$ being infinite or finite, not its exact value, and infinite $\nu$-complexity is equivalent to (joint) indecomposability.

\begin{example}
The homogeneous series $f=\sum_{i \ge 1} x_i^2$ is indecomposable.
\end{example}

\begin{example}
Let $f_1 = \sum_{i\geq 1} x_i$ and $f_2 = \sum_{i\geq 1} x_i^2$.  Since $f_1$ and $f_2$ have different degrees, any homogeneous linear combination $\alpha_1f_1+\alpha_2f_2$ must have either $\alpha_1=0$ or $\alpha_2=0$.  Since $f_1$ is linear and $f_2$ is indecomposable (by the previous example), it follows that $f_1$ and $f_2$ are jointly indecomposable.
\end{example}

\begin{example}\label{ex:dth powers}
For each $d\geq 1$, let $f_d=\sum_{i\geq 1} x_i^d$.  It turns out that each $f_d$ is indecomposable.  Since the $f_d$ all have different degrees, any homogeneous linear combination of the $f_d$ will be a scalar multiple of one of the $f_d$.  It follows that the infinite set $\{f_1,f_2,\dots\}$ is jointly indecomposable.
\end{example}

\begin{example}
Let $f_1=\sum_{i \ge 1} x_i^2$ and $f_2 = \sum_{i \ge 1} i x_i^2$. One can show that $f_1$ and $f_2$ are jointly indecomposable.
\end{example}

\subsection{The main theorem}

We introduced homogeneous series with the hope that we could replace the asymptotic form of the Ananyan--Hochster principle with something more precise. We now see our hopes realized.

\begin{theorem} \label{thm:algindseries}
Any collection of jointly indecomposable homogeneous series of positive degree is algebraically independent.
\end{theorem}

For example, the $d$th power sums $f_d$ from Example~\ref{ex:dth powers} are algebraically independent.  The theorem is equivalent to two other noteworthy statements that we give as corollaries.

\begin{corollary}
Let $\{g_i\}_{i \in \cI}$ be a maximal set of jointly indecomposable homogeneous series of positive degree, where $\cI$ is an index set. Given any homogeneous series $f$ there exist distinct indices $i_1, \ldots, i_n \in \cI$ and a polynomial $F\in \bC[X_1,\dots,X_n]$ such that $f=F(g_{i_1}, \ldots, g_{i_n})$. Moreover, this expression is unique up to applying a permutation to $i_1, \ldots, i_n$ and the inverse permutation to $F$.
\end{corollary}

To state our second corollary, we must introduce a new object. A {\bf bounded degree series} is a finite sum of homogeneous series, of possibly different degrees. Let $\bR$ be the set of all bounded degree series. As any sum or product of bounded degree series is again a bounded degree series, we see that $\bR$ forms a commutative ring. Theorem~\ref{thm:algindseries} translates into the following precise description of the structure of $\bR$:

\begin{corollary}\label{cor:inv limit}
The ring $\bR$ is abstractly a polynomial ring. Precisely, let $\{g_i\}_{i \in \cI}$ be a maximal set of jointly indecomposable homogeneous series of positive degree. Then $\bR$ is isomorphic to the polynomial ring $\bC[X_i]_{i \in \cI}$ with variables indexed by $\cI$. The isomorphism takes a polynomial $F(X_i)_{i \in \cI}$ to the bounded degree series $F(g_i)_{i \in \cI}$ obtained by substituting $g_i$ for $X_i$ for all $i$.
\end{corollary}

This final corollary meets our goal of finding a precise form of the Ananyan--Hochster principle: it shows that jointly indecomposable homogeneous series are literally independent variables (up to an isomorphism). 

\begin{remark}
The index set $\cI$ in Corollary~\ref{cor:inv limit} is always of uncountable cardinality, and so there are uncountably many variables in the ring $\bC[X_i]_{i \in \cI}$.
\end{remark}

\subsection{Back to polynomials}\label{subsec:back to polynomials}
While this idealized Ananyan--Hochster principle for homogeneous series provides some helpful conceptual clarity, we would really like to use it to also derive results for polynomials, such as those in \S\ref{s:principle}.  Unfortunately, the most direct approach to doing so does not really work.

For example, suppose we wanted to try to prove Theorem~\ref{thm:reg}. Thus we start with a sequence $f_{1,i},\ldots,f_{r,i}$ of tuples of polynomials indexed by $i \in \bN$ with $\nu$-complexity tending to infinity, and we would like to show that the $i$th tuple is a regular sequence for $i \gg 0$. The most direct approach would be to somehow define homogeneous series $f_1, \ldots, f_r$ by taking the limit of $f_{1,i}, \ldots, f_{r,i}$ as $i \to \infty$, then apply the results of this section to conclude that $f_1, \ldots, f_r$ forms a regular sequence, and finally argue that this implies $f_{1,i},\ldots,f_{r,i}$ form a regular sequence for all $i\gg 0$. The problem with this approach is in forming the limit: given an arbitrary sequence $(g_i)_{i \ge 1}$ of polynomials, there is not necessarily a reasonable limiting homogeneous series. (For example, consider the case where $g_i$ is simply $x_i$ itself.)

It turns out that one \emph{can} apply Corollary~\ref{cor:inv limit} to the study of polynomials (see the proof of Theorem~\ref{thm:stillmanR}), but the connection is much more subtle than the approach outlined above. One of the great advantages of our second idealized Ananyan--Hochster principle, discussed in the following section, is that it does directly connect to polynomials.

\subsection{Proof of the main theorem}\label{subsec:proof 4}
The proof of Theorem~\ref{thm:algindseries} we give in \cite{stillman} is short and entirely elementary. We now give some indication of the main idea. Consider a hypothetical algebraic relation
\begin{displaymath}
F(f_1, \ldots, f_r) = 0
\end{displaymath}
where $F(X_1, \ldots, X_r)$ is a polynomial and $f_1, \ldots, f_r$ are jointly indecomposable homogeneous series. Now differentiate this equation with respect to some variable, say $x_i$. (We differentiate homogeneous series termwise.) By the chain rule, we obtain
\begin{equation}\label{eq:neweqn}
\sum_{j=1}^r F_j(f_1, \ldots, f_r) \frac{\partial f_j}{\partial x_i} = 0,
\end{equation}
where $F_j=\frac{\partial F}{\partial X_j}$ is the $j$th partial derivative of $F(X_1, \ldots, X_r)$. This is an algebraic relation among the $2r$ homogeneous series $f_1, \ldots, f_r, \partial f_1/\partial x_i, \ldots, \partial f_r/\partial x_i$.  Its total degree is one less than that of the original relation. Thus, arguing inductively on the total degree, we can assume that this relation is trivial. (There is a subtlety here: the $2r$ series in this relation may no longer be jointly indecomposable. However, one can express them in terms of some set of jointly indecomposable series, and then, after making these substitutions, the resulting algebraic relation is trivial.)  As this holds for all choices of the variable $x_i$, one can conclude that the original relation is trivial. We refer the reader to \cite[\S2]{stillman} for the details.

\section{Homogeneous germs} \label{s:germs}
We now describe a second way to construct a limiting setting in which the Ananyan--Hochster principle becomes exact, using techniques from model theory. Compared to the approach of the previous section, this approach is more technical, but it has the advantage of applying more directly to ordinary polynomials.

\subsection{Homogeneous sequences}

As the Ananyan--Hochster principle is concerned with the limiting behavior of homogeneous polynomials of fixed degree, it is natural to introduce the following definition:

\begin{definition}
A {\bf homogeneous sequence} is a sequence $f_{\bullet} = (f_1,f_2,f_3,\dots)$ of homogeneous polynomials, all of the same degree, which we refer to as the {\bf degree} of $f_{\bullet}$.
\end{definition}

Thus we are interested in the limiting behavior of homogeneous sequences. The primary technical problem here is the same problem that one encounters when studying limits of any kind: they need not be defined. In other words, a homogeneous sequence may exhibit one kind of limiting behavior along one subsequence, and another along another. For example, consider the homogeneous sequence $f_{\bullet}$ given by
\begin{equation} \label{eq:exf}
f_i = \begin{cases}
x_2^2+\cdots+x_i^2 & \text{if $i$ is even} \\
x_1^2 & \text{if $i$ is odd} \end{cases}
\end{equation}
This has $\nu(f_i)=1$ for $i$ odd and $\nu(f_i)=i$ for $i$ even; thus on odd integers, $f_{\bullet}$ decomposes uniformly, but on the even integers it does not. This is but one example of what can go wrong. Here is another: one can have a homogeneous series $f_{\bullet}$ with $\nu(f_i) \le 2$ for all $i$, so that each $f_i$ can be written as a function of two lower degree polynomials; however, it could be that for $i$ odd these two polynomials each have degree~1, while for $i$ even they each have degree~2. Thus, writing $f_i=F_i(g_{1,i}, g_{2,i})$, the sequences $g_{1,\bullet}$ and $g_{2,\bullet}$ are not homogeneous sequences, since they do not have constant degree. In other words, decomposing a uniformly decomposable homogeneous sequence might take us outside of the world of homogeneous sequences.

These issues may seem like mere annoyances, but when carrying out complicated operations on homogeneous sequences they compound and create so much bookkeeping as to obscure the main ideas. It is therefore highly desirable to come up with a way to have well-defined limiting behavior.

Let us return to the example of \eqref{eq:exf}, where the sequence $f_{\bullet}$ has different behavior on even and odd integers. Imagine that the even integers converged to some point $p$, and the odd integers to some point $q$. We could then say that $f_{\bullet}$ exhibits one type of behavior near $p$, and another type near $q$. This suggests that we should try to work with our homogeneous sequences locally with respect to some topology on the index set $\bN$.

Of course, this raises the question of which topology to use. In fact, there is a ``best'' choice: the Stone--\v{C}ech compactification. The Stone--\v{C}ech compactification $\beta X$ of a topological space $X$ is the universal compact Hausdorff space that admits a continuous map from $X$. We are mainly interested in the Stone--\v{C}ech compactification $\beta \bN$ of the discrete topological space $\bN$. It is not difficult to see that $\beta \bN$ is totally disconnected and that $\bN$ is a dense subset of $\beta \bN$. This is essentially all that one needs to know of $\beta \bN$ for our discussion.

Since $\beta \bN$ is compact, every sequence in it has a limit point. Thus, in the context of \eqref{eq:exf}, one can essentially do what we had hoped, and work with the limiting points $p$ and $q$; this is only slightly incorrect, due to the fact that the even and odd integers will not have unique limiting points.

This discussion suggests that we should study the local behavior of homogeneous sequences at some point of $\beta \bN$. This turns out to work very well, and is what we will do in the following subsection. However, before beginning that discussion we give a more direct description of the points of $\beta \bN$ that does not rely on topology.

\begin{definition}
  Let $X$ be a set. An {\bf ultrafilter} on $X$ is a collection $\cU$ of subsets of $X$ satisfying the following conditions:
\begin{itemize}
\item Given $A \subset B \subset X$ with $A \in \cU$ we have $B \in \cU$.
\item Given $A, B \in \cU$ we have $A \cap B \in \cU$.
\item Given $A \subset X$, either $A \in \cU$ or $X \setminus A \in \cU$.
\item The empty set does not belong to $\cU$.
\end{itemize}
Given $x \in X$, the collection $\cU$ of subsets of $X$ containing $x$ is an ultrafilter on $X$, called the {\bf principal ultrafilter} at $x$.
\end{definition}

If $x$ is a point in the Stone--\v{C}ech compactification $\beta X$ (regarding $X$ as a discrete space) then the collection
\begin{displaymath}
\{ U \cap X \mid \text{$U$ is an open neighborhood of $x$ in $\beta X$} \}
\end{displaymath}
is an ultrafilter on $X$. In fact, this gives a bijection between $\beta X$ and the set of ultrafilters on $X$, with elements of $X \subset \beta X$ corresponding to principal ultrafilters. Thus, instead of working with a point of $\beta \bN$, we can work with an ultrafilter on $\bN$, and this is what we actually do. Principal ultrafilters do not lead to an interesting theory, so we will only use non-principal ultrafilters. We note, however, that the existence of non-principal ultrafilters relies on the axiom of choice; in particular, one cannot write down an example of one explicitly.

\subsection{Homogeneous germs}
Fix a non-principal ultrafilter on $\bN$. We let $\ast$ denote the corresponding point of $\beta \bN\setminus \bN$ and we refer to subsets in the ultrafilter as ``neighborhoods of $\ast$.''

\begin{definition} \label{defn:decompnear}
We say that a homogeneous sequence $f_{\bullet}$ is {\bf decomposable near $\ast$} if there exist homogeneous sequences $g_{1,\bullet}, \ldots, g_{n,\bullet}$ of degree strictly less than that of $f_{\bullet}$, and polynomials $F_{\bullet}$, such that $f_i=F_i(g_{1,i}, \ldots, g_{n,i})$ holds for all $i$ in some neighborhood of $\ast$.
\end{definition}

\begin{example}
Let $f_{\bullet}$ be the homogeneous sequence given by
\begin{displaymath}
f_i=\begin{cases}
x_1^2+\cdots+x_i^2 & \text{if $i$ is even} \\
x_1^2 & \text{if $i$ is odd} \end{cases}.
\end{displaymath}
This sequence may or may not be decomposable near $\ast$, depending on what point of $\beta \bN$ we have chosen for $\ast$. If the even numbers form a neighborhood of $\ast$ then $f_{\bullet}$ is indecomposable near $\ast$; if the odd numbers form a neighborhood of $\ast$ then $f_{\bullet}$ is decomposable near $\ast$. Exactly one of these two possibilities holds by the axioms for ultrafilters.
\end{example}

Whether or not $f_{\bullet}$ is decomposable near $\ast$ only depends on its local behavior near $\ast$. From analysis and sheaf theory, we know that to study local behavior we should consider germs. We therefore make the following definition:

\begin{definition}
We define an equivalence relation $\sim$ on homogeneous sequences as follows: we declare $f_{\bullet} \sim g_{\bullet}$ if $f_i=g_i$ for all $i$ in some neighborhood of $\ast$. A {\bf homogeneous germ} is an equivalence class of homogeneous sequences. For a homogeneous sequence $f_{\bullet}$, we let $[f_{\bullet}]$ denote the homogeneous germ that it defines.
\end{definition}

Homogeneous germs are well-behaved, from a formal point of view. Indeed, suppose that $[f_{\bullet}]$ and $[g_{\bullet}]$ are homogeneous germs. We can then define their product $[f_{\bullet}][g_{\bullet}]$ to be the homogeneous germ $[f_{\bullet} g_{\bullet}]$, where $f_{\bullet} g_{\bullet}$ is defined pointwise. One easily verifies that this is indeed well-defined. Similarly, we can define addition of homogeneous germs of the same degree. We have the following interesting observation:

\begin{proposition}
The set of homogeneous germs of degree~$0$ forms a field, under the addition and multiplication laws just defined.
\end{proposition}

\begin{proof}
It is an easy exercise to see that the addition and multiplication laws endow the set of homogeneous germs of degree~0 with the structure of a commutative ring. Let us explain why it is a field. Thus suppose that $[\alpha_{\bullet}]$ is a non-zero homogeneous germ of degree~0. We must show that it has a reciprocal.

What does it mean that $[\alpha_{\bullet}]$ is non-zero? It means that it is not equal to the zero homogeneous germ, which by definition, is the homogeneous germ that is the identity for addition. Clearly, this is the homogeneous germ $[0_{\bullet}]$ where $0_{\bullet}$ is the homogeneous sequence given by $0_i=0$ for all $i$. Our hypothesis is thus $[\alpha_{\bullet}] \ne [0_{\bullet}]$. By definition, this means that $\alpha_{\bullet}$ and $0_{\bullet}$ are inequivalent under $\sim$. Thus, if $U$ denotes the set of indices $i \in \bN$ for which $\alpha_i=0$, then $U$ is not a neighborhood of $\ast$. But, by the axioms of ultrafilters, this means its complement $V=\bN \setminus U$ is a neighborhood of $\ast$. In other words, $\alpha_i$ is a non-zero complex number for all $i$ in the neighborhood $V$ of $\ast$.

Now, define a homogeneous sequence $\beta_{\bullet}$ by
\begin{displaymath}
\beta_i=\begin{cases}
\alpha_i^{-1} & \text{for $i \in V$} \\
1 & \text{for $i \in U$} \end{cases}.
\end{displaymath}
Then $\alpha_i \beta_i=1$ for all $i$ in the neighborhood $V$, and so $[\alpha_{\bullet}] [\beta_{\bullet}]=[1_{\bullet}]$, where $1_{\bullet}$ is the homogeneous sequence with $1_i=1$ for all $i$. As $[1_{\bullet}]$ is clearly the multiplicative unit for homogeneous germs, we see that $[\beta_{\bullet}]$ is the reciprocal of $[\alpha_{\bullet}]$, and so the proposition follows.
\end{proof}

Let ${}^*\bC$ be the field of homogeneous germs of degree~0. This field is the {\bf ultrapower} of the field $\bC$ of complex numbers, sometimes called the field of {\it hypercomplex numbers}. It is an enormous field---it is an algebraically closed extension of $\bC$ of uncountable degree---and hard to really picture. However, it is easy to work with ${}^*\bC$ in a formal sense: its elements are simply sequences of complex numbers up to the equivalence relation $\sim$.

The significance of ${}^*\bC$ to the present discussion is that it is the appropriate field of scalars for working with homogeneous germs. Indeed, if $[\alpha_{\bullet}] \in {}^*\bC$ and $[f_{\bullet}]$ is a homogeneous germ of degree $d$ then $[\alpha_{\bullet}] [f_{\bullet}]=[\alpha_{\bullet} f_{\bullet}]$ is again a homogeneous germ of degree $d$. Thus the set of homogeneous germs of degree $d$ is a vector space over ${}^*\bC$. Furthermore, we find that any polynomial expression in homogeneous germs with coefficients in ${}^*\bC$, and that is appropriately homogeneous, is again a homogeneous germ. The following definition is therefore forced onto us by analogy with our previous ones:

\begin{definition} \label{defn:decompgerm}
A homogeneous germ $[f_{\bullet}]$ is {\bf $n$-decomposable} if there exist homogeneous germs $[g_{1,\bullet}], \ldots, [g_{n,\bullet}]$ and a polynomial $F(X_1, \ldots, X_n)$ with coefficients in ${}^*\bC$ such that $[f_{\bullet}]=F([g_{1,\bullet}], \ldots, [g_{n,\bullet}])$; it is {\bf indecomposable} if it fails to be $n$-decomposable for all $n$.  A collection of homogeneous germs of positive degree is {\bf jointly indecomposable} if every non-trivial homogeneous ${}^*\bC$-linear combination is indecomposable.
\end{definition}

The following proposition, which we leave as an exercise, connects the two notions of decomposability introduced in this section.

\begin{proposition}
Let $f_{\bullet}$ be a homogeneous sequence. Then $[f_{\bullet}]$ is decomposable (in the sense of Definition~\ref{defn:decompgerm}) if and only if $f_{\bullet}$ is decomposable near $\ast$ (in the sense of Definition~\ref{defn:decompnear}).
\end{proposition}

This proposition, while entirely formal, is conceptually important because it shows that the decomposability of $f_{\bullet}$ near $\ast$ can detected from the germ $[f_{\bullet}]$ and germ-level constructions (germ addition, multiplication, and scalar multiplication). In other words, one does not have to ``look inside'' of $[f_{\bullet}]$ to determine if $f_{\bullet}$ is decomposable near $\ast$.

\subsection{The main theorem}

The main theorem, and its corollaries, in the setting of homogeneous germs is entirely analogous to that for homogeneous series, and once again realizes an idealized form of the Ananyan--Hochster principle.

\begin{theorem}
Any collection of jointly indecomposable homogeneous germs of positive degree is algebraically independent (relative to the coefficient field ${}^*\bC$).
\end{theorem}
The proof of this theorem is nearly identical to the proof outlined in \S\ref{subsec:proof 4}.

\begin{corollary}
Let $\{[g_{i,\bullet}]\}_{i \in \cI}$ be a maximal set of jointly indecomposable homogeneous germs of positive degree, where $\cI$ is an index set. Given any homogeneous germ $[f_{\bullet}]$ there exist distinct indices $i_1, \ldots, i_n \in \cI$ and a polynomial $F$ (with coefficients in ${}^*\bC$) such that $[f_{\bullet}]=F([g_{i_1,\bullet}], \ldots, [g_{n,\bullet}])$. Moreover, this expression is unique up to applying a permutation to $i_1, \ldots, i_n$ and the inverse permutation to $F$.
\end{corollary}

A {\bf bounded degree germ} is a finite sum of homogeneous germs, of possibly varying degrees. Let $\bS$ be the set of all bounded degree germs. This is a graded ring, and contains the field ${}^*\bC$ as its degree~0 piece.

\begin{corollary}\label{cor:ultra poly}
The ring $\bS$ is a polynomial ring (over ${}^*\bC$). Precisely, let $\{[g_{i,\bullet}]\}_{i \in \cI}$ be a maximal set of jointly indecomposable homogeneous germs of positive degree. Then $\bS$ is isomorphic to the polynomial ring ${}^*\bC[X_i]_{i \in \cI}$ with variables indexed by $\cI$. The isomorphism takes a polynomial $F(X_i)_{i \in \cI}$ to the bounded degree germ $F([g_{i,\bullet}])_{i \in \cI}$ obtained by substituting $[g_{i,\bullet}]$ for $X_i$ for all $i$.
\end{corollary}

Much like Corollary~\ref{cor:inv limit}, this corollary meets our goal of finding a precise form of the Ananyan--Hochster principle.

\begin{remark}
Ultraproducts were previously used in commutative algebra to establish uniform bounds for polynomials in a fixed number of variables but with varying coefficient field~\cite{van-den-Dries-schmidt}. Our results, which also allow for varying field, are novel in that they allow the number of variables to grow. Ultrafilters have also been used to connect results in characteristic~$p$ and in characteristic~$0$, see~\cite{schoutens}.
\end{remark}

\section{Homogeneous germs and regular sequences} \label{s:reggerm}

We now explain how to use the idealized Ananyan--Hochster principle for homogeneous germs to deduce Theorem~\ref{thm:reg}, which is an instance of the Ananyan--Hochster principle for polynomials. The same method can be used to deduce other instances.  For clarity, we rephrase Theorem~\ref{thm:reg} as follows:

\begin{theorem}\label{thm:high strength reg seq}
Given $d,r \in \bN$, there exists $N \in \bN$ with the following property: if $f_1, \ldots, f_r$ are homogeneous polynomials of degrees $\le d$ with $\nu(f_1,\dots,f_r) > N$ then $f_1,\dots,f_r$ forms a regular sequence.
\end{theorem}

\begin{proof}[Sketch of proof]
Consider a sequence $(f_{1,i}, \ldots, f_{r,i})_{i \ge 1}$ of tuples of homogeneous polynomials of degrees $\le d$ with $\nu(f_{1,i}, \ldots, f_{r,i})$ tending to infinity with $i$. If $f_{1,i},\ldots,f_{r,i}$ forms a regular sequence for all $i\gg 0$ then we are done.  If this is not the case, then, by passing to a subsequence, we can assume that $f_{1,i},\ldots,f_{r,i}$ {\em fails} to form a regular sequence for all $i$.  We will show that this latter possibility cannot occur.

Consider the homogeneous germs $[f_{1,\bullet}], \ldots, [f_{r,\bullet}]$. These are jointly indecomposable: indeed, an $n$-decomposition of them would yield an $n$-decomposition of $(f_{1,i},\ldots,f_{r,i})$ for all $i$ in a neighborhood of $\ast$, thus bounding $\nu(f_{1,i},\ldots,f_{r,i})$ in this neighborhood, a contradiction. We may thus assume that $[f_{1,\bullet}], \dots, [f_{r,\bullet}]$ are part of a maximal set of jointly indecomposable homogeneous germs of positive degree. By Corollary~\ref{cor:ultra poly}, there is an isomorphism of $\bS$ with a polynomial ring ${}^*\bC[X_i]_{i \in \cI}$ where $[f_{1,\bullet}], \dots, [f_{r,\bullet}]$ are mapped to distinct variables.  Since distinct variables in a polynomial ring form a regular sequence, and since this property is invariant under ring isomorphisms, it follows that $[f_{1,\bullet}], \dots, [f_{r,\bullet}]$ form a regular sequence in $\bS$.

To complete the proof, it now suffices to prove the following general statement: if homogeneous germs $[g_{1,\bullet}], \dots, [g_{r,\bullet}]$ form a regular sequence, then the polynomials $g_{1,i},\dots,g_{r,i}$ form a regular sequence in some neighborhood of $\ast$. We do this in \cite[Corollary~4.10]{stillman}. The proof crucially uses Corollary~\ref{cor:ultra poly}, but is otherwise straightforward commutative algebra. However, the details would take us too far afield.
\end{proof}

\begin{remark}
The proof of Theorem~\ref{thm:high strength reg seq} given above is based on~\cite[\S4]{stillman}, and it has a very different flavor from the proof given in~\cite{ananyan-hochster}.  The proof in~\cite{ananyan-hochster} has a ``bottom up'' structure, using a six-fold induction to prove increasingly nice properties about the polynomials $f_1,\dots,f_r$ as the $\nu$-complexity grows. In particular, it relies on a number of different instances of the Ananyan--Hochster principle.  By contrast, the proof outlined above has a ``top down'' structure, where the key insight lies in understanding homogeneous germs with infinite $\nu$-complexity, and then the specific property of being a regular sequence is descended to the case of large $\nu$-complexity.
\end{remark}

\section{From Hilbert's Syzygy Theorem to Stillman's Conjecture} \label{s:stillman}
\subsection{Algebra review}
To state Hilbert's Syzygy Theorem and Stillman's Conjecture, we need to review some algebraic notions.  Let $S=\bC[x_1,\dots,x_n]$ be the polynomial ring in $n$ variables.  Given polynomials $f_1,\dots,f_r$, the {\bf ideal} $(f_1,\dots,f_r)$ is the subset of $S$ consisting of all combinations $\sum_{i=1}^r a_if_i$, where the coefficients $a_i$ are allowed to be arbitrary polynomials in $S$.  Generally, the interesting properties of a collection of polynomials $f_1,\dots,f_r$ depend only on the ideal $(f_1,\dots,f_r)$, and not on the specific choice of generators.

An $S$-module is a ring-theoretic analogue of a vector space.  In particular, an $S$-module is an abelian group $M$ together with a scalar multiplication $S\times M\to M$ that satisfies certain basic axioms like distributivity.  The simplest modules are the {\bf free modules} $S^m$ and those behave very analogously to vector spaces.  The free module $S^m$ is the set of sequences $(a_1,a_2,\dots,a_m)$ with $a_i\in S$; addition is termwise and scalar multiplication is $f\cdot (a_1,\dots,a_m)=(fa_1,\dots,fa_m)$ for any $f\in S$.

Hilbert's Syzygy Theorem compares arbitrary $S$-modules (which can be quite complicated) with free $S$-modules.  The key definition is that of a {\bf free resolution}: this is a diagram
\begin{equation}\label{eqn:free res}
S^{b_0}\xleftarrow{\ \phi_1 \ }S^{b_1}\xleftarrow{\ \phi_2 \ } S^{b_2} \xleftarrow{\ \phi_3 \ } \cdots
\end{equation}
where each $\phi_i$ is an $S$-module homomorphism, and where the kernel of $\phi_1$ equals the image of $\phi_2$, the kernel of $\phi_2$ equals the image of $\phi_3$, and so on. If $M$ equals the cokernel of $\phi_1$, then this is said to be a {\bf free resolution of $M$}. In this case, the first two terms of the resolution provide a presentation for $M$, with $b_0$ generators and $b_1$ relations. The numbers $b_2,b_3,\dots$ are more subtle: $b_2$ is something like the number of secondary relations (the relations among the relations) and so on. In astronomy, the word {\bf syzygy} refers to a conjunction, often of astrological bodies; in algebra, it is used to refer to these secondary relations, tertiary relations, and so on.\footnote{While the numbers $b_i$ are not unique as we have defined them, they can be made unique for graded modules through the notion of a {\bf minimal free resolution}.}

\begin{example}
Let $S=\bC[x_1,x_2]$ and let $M$ be the $S$-module $S/(x_1^2, x_1x_2)$. One has the following free resolution of $M$:
\[
S^{1}\xleftarrow{\ \phi_1 \ }S^{2}\xleftarrow{\ \phi_2 \ } S^{1} \xleftarrow{\ \phi_3 \ } 0 \xleftarrow{\ \phi_4 \ } 0 \xleftarrow{\ \phi_5 \ } \cdots
\]
where the morphisms are represented by  matrices as follows
\[
\phi_1 = \begin{bmatrix} x_1^2 & x_1x_2\end{bmatrix} \text{ and } \phi_ 2 = \begin{bmatrix} -x_2\\ x_1 \end{bmatrix}.
\]
An elementary computation confirms that the kernel of $\phi_1$ equals the image of $\phi_2$, and that the kernel of $\phi_2$ is zero.
\end{example}

Of the many invariants one can extract from free resolutions, one of the most important is projective dimension: the {\bf projective dimension} of an $S$-module $M$ is the minimal $p$ such that $M$ has a free resolution that terminates (i.e., is zero) after $p$ steps:
\[
S^{b_0}\xleftarrow{\ \phi_1 \ }S^{b_1}\xleftarrow{\ \phi_2 \ } S^{b_2} \xleftarrow{\ \phi_3 \ } \cdots \xleftarrow{\ \phi_p \ } S^{b_p} \gets 0 \gets 0 \gets \cdots
\]
Hilbert's Syzygy Theorem states that, for the polynomial ring $S$, every module has finite projective dimension.  Even better, the projective dimension is bounded above by the number of variables.

\begin{theorem}[Hilbert's Syzygy Theorem]
Let $S=\bC[x_1,\dots,x_n]$.  The projective dimension of any finitely generated $S$-module is at most $n$.  In particular, if $f_1,\dots,f_r\in S$, then the projective dimension of $S/(f_1,\dots,f_r)$ is at most $n$.
\end{theorem}

Free resolutions are computable objects (in a very strong sense: these objects can be computed by the computer algebra system {\em Macaulay2}~\cite{M2}) from which one can obtain many useful invariants of the sequence $f_1,\dots,f_r$;  see Remark~\ref{rmk:applications of syzygies} below.  The projective dimension provides one measure of the size of the minimal free resolution of $f_1,\dots,f_r$, and it is related to the computational complexity of answering certain questions about the sequence $f_1,\dots,f_r$.

\subsection{Stillman's Conjecture}
If $n$ is very large, then the bound on projective dimension given by Hilbert's Syzygy Theorem might be very far from optimal.  For instance, if we had $4$ cubic polynomials in $10^{100}$ variables, Hilbert's bound would say that the projective dimension is at most $10^{100}$. It is natural to ask if we can do better.  Stillman first proposed this type of question, asking whether there is an a priori upper bound on the projective dimension of an ideal that depends on the number of polynomials and their degrees, but which is insensitive to the number of variables:

\begin{conjecture}[Stillman's Conjecture]\label{conj:Stillman}
Let $d_1,\dots,d_r$ be positive integers. There exists a positive integer $B(d_1,\dots,d_r)$ satisfying the following condition: if $f_1,\dots,f_r$ are any homogeneous polynomials with $\deg(f_i)=d_i$ in a polynomial ring $S=\bC[x_1,x_2,\dots,x_n]$ (for any $n$), then the projective dimension of $S/(f_1,\dots,f_r)$ is at most $B(d_1,\dots,d_r)$.
\end{conjecture}
Stillman, who is one of the authors of the computational program {\em Macaulay2}~\cite{M2}, was interested in this question due to its potential connection with Gr\"obner basis algorithms. These algorithms are essential to symbolic computation in algebra and algebraic geometry, but they are infamous for their complexity: in the worst case, the run time grows doubly exponentially in the number of variables.  However, a bound on projective dimension---such as the one in Conjecture~\ref{conj:Stillman}---might allow for alternate computational techniques in special circumstances.

\begin{remark}\label{rmk:applications of syzygies}
The study of free resolutions has applications to a huge array of topics related to algebra.  For instance, free resolutions were used by Stanley in his proof of the Upper Bound Conjecture in combinatorics~\cite{stanley}.  Also, starting with highly influential conjectures of Mark Green~\cite{green-koszul1,green-koszul2}, which were later largely proven by Voisin~\cite{voisin-3,voisin-2,voisin-1}, free resolutions have been used to understand subtle geometric properties of algebraic curves.  One recent result in this vein is Ein and Lazarsfeld's 2015 proof of the Gonality Conjecture from~\cite{green-lazarsfeld}.  The {\bf gonality} of a smooth projective curve $C$ is the minimal degree of a map $C\to \bP^1$, and the Gonality Conjecture relates the gonality of a curve to its free resolution.
The interested reader should see~\cite{EL}.
\end{remark}

\subsection{Special cases of Stillman's Conjecture}
When $r=1$ the projective dimension of $S/(f_1)$ is at most $1$ and when $r=2$ the projective dimension of $S/(f_1,f_2)$ is at most $2$.  When $r=3$ things become much more complicated, as illustrated by the following theorem.

\begin{theorem}\label{thm:bruns}
Fix any $n\geq 1$ and let $S=\bC[x_1,\dots,x_n]$.  There exist polynomials $f_1,f_2,f_3$ such that the projective dimension of $S/(f_1,f_2,f_3)$ is $n$ (i.e., the maximum value allowed by Hilbert's Syzygy Theorem).
\end{theorem}
Variants of this theorem were proven by Burch~\cite{burch}, Kohn~\cite{kohn}, and Bruns~\cite{bruns}.  Even if one bounds the degrees of the polynomials, the projective dimension can be surprisingly large for three polynomials.

\begin{theorem}[\cite{bmnsss}]
Fix $d$ and any $n\gg d$.  There exist degree $d$ polynomials $f_1,f_2,f_3\in S=\bC[x_1,\dots,x_n]$ such that the projective dimension of $S/(f_1,f_2,f_3)$ is at least $\sqrt{d}^{\sqrt{d}-1}$.
\end{theorem}

Theorem~\ref{thm:bruns} shows that no bound on projective dimension exists solely in terms of the number of polynomials; the theorem of \cite{bmnsss} shows that any positive answer to Stillman's Conjecture would grow quite quickly in $d$.

\begin{remark}
In the case where $f_1,\dots,f_r\in S=\bC[x_1,\dots,x_n]$ define a smooth subvariety of $\bP^{n-1}$, a theorem of Faltings shows that the projective dimension of $S/(f_1,\dots,f_r)$ is at most $3r$~\cite{faltings}.
\end{remark}

\section{The Ananyan--Hochster principle implies Stillman's Conjecture}\label{s:proof of stillman}

We now explain how to use one instance of the Ananyan--Hochster principle, namely Theorem~\ref{thm:reg} (which is the same as Theorem~\ref{thm:high strength reg seq}) to prove Stillman's Conjecture. We begin with the following elementary observation, which shows that we can write a given collection of polynomials in terms of polynomials with high $\nu$-complexity, with great flexibility:

\begin{proposition}
Let $d,r \in \bN$ be given, together with a function $N \colon \bN \to \bN$. Then there exist $s \in \bN$ with the following property: given any homogeneous polynomials $f_1, \ldots, f_r$ of degrees $\le d$ there exist homogeneous polynomials $g_1, \ldots, g_s$ of degrees $\le d$ with $\nu(g_1,\ldots,g_s) > N(s)$ such that each $f_i$ can be written as $F_i(g_1,\ldots,g_s)$ for some polynomial $F_i$.
\end{proposition}

\begin{proof}
To produce the $g_j$'s, we execute the following algorithm:
\begin{itemize}
\item[(A)] Initialize with $t=r$ and $g_i=f_i$ for $1 \le i \le r$.
\item[(B)] If $\nu(g_1, \ldots, g_t) > N(t)$ halt with output $(g_1, \ldots, g_t)$.
\item[(C)] Otherwise, make a linear change of variables in the $g_j$'s so that $g_t$ is $N(t)$-decomposable, write $g_t=P(g'_1, \ldots, g'_{N(t)})$, replace $(g_1, \ldots, g_t)$ with $(g_1, \ldots, g_{t-1}, g'_1, \ldots, g'_{N(t)})$ and return to step (B).
\end{itemize}
We must explain why this algorithm halts, and why the length of the final list can be bounded in terms of $d$, $r$, and $N$.

Let $d_i$ be the degree of $g_i$. In step (C), note that each $g'_i$ has degree strictly less than $d_t$. Thus $(d_1, \ldots, d_{t-1}, d'_1, \ldots, d'_{N(t)})$ is strictly smaller than $(d_1, \ldots, d_t)$, if we sort the tuples from largest to smallest and compare lexicographically. Since tuples of non-negative integers under lexicographic order is a well-ordered set, it follows that the procedure terminates.

To bound the length of the final tuple, we proceed inductively. Suppose that on the first pass through we reach step (C). Since the new tuple $(g_1, \ldots, g_{r-1}, g'_1, \ldots, g'_{N(r)})$ is smaller than the initial tuple, we can bound the length of the output purely in terms of $d$, $N$, and the length of this new starting tuple, i.e., $r-1+N(r)$; this is the inductive hypothesis. Thus the length can be bounded simply in terms of $d$, $r$, and $N$, as required.
\end{proof}

\begin{proof}[Proof of Stillman's Conjecture]
Fix $d \in \bN$. For $s \in \bN$, let $N(s)$ be the bound produced by Theorem~\ref{thm:high strength reg seq} with $r=s$; thus, if $g_1, \ldots, g_s$ are homogeneous polynomials of degrees $\le d$ with $\nu(g_1, \ldots, g_s) \ge N(s)$ then $g_1, \ldots, g_s$ is a regular sequence. Now fix $r \in \bN$, and let $s$ be as in the above proposition, with respect to $d$, $r$, and $N(-)$. 

If $f_1, \ldots, f_r$ are homogeneous polynomials of degrees $\le d$, then we can find homogeneous polynomials $g_1, \ldots, g_s$ of degrees $\le d$ and with $\nu(g_1, \ldots, g_s)>N(s)$ such that each $f_i$ can be written as a polynomial $f_i=F_i(g_1,\dots,g_s)$ in the $g_j$'s; thus, each $f_i$ belongs to the subalgebra $\bC[g_1, \ldots, g_s]$ of $\bC[x_1, \ldots, x_n]$. The lower bound on the $\nu$-complexity of the $g_j$'s, together with our choice of $N$, ensures that $g_1, \ldots, g_s$ forms a regular sequence (by Theorem~~\ref{thm:high strength reg seq}).

Let $S=\bC[x_1, \ldots, x_n]$ and let $R \subset S$ be the subalgebra $\bC[g_1, \ldots, g_s]$. Let $I \subset S$ and $J \subset R$ be the ideals generated by the $f$'s. Since $g_1,\dots,g_s$ are regular sequence they are also algebraically independent, and thus $R$ is abstractly a polynomial ring in $s$ variables, and the Hilbert Syzygy Theorem implies that $R/J$ has projective dimension $\le s$ as an $R$-module. So we have a free resolution $F_{\bullet} \to R/J$ of length at most $s$. We now come to the final step, which is standard if a bit technical: since the $g_j$'s form a regular sequence on $S$, we have that $S$ is a free $R$-module (see~\cite[Proposition~2.2.11]{bruns-herzog}), and thus the functor $- \otimes_R S$ is exact.  It follows that $F_{\bullet} \otimes_R S \to (R/J) \otimes_R S$ is a free resolution, and so $(R/J) \otimes_R S$ has projective dimension at most $s$. Since $(R/J) \otimes_R S$ is isomorphic to $S/I$, the result follows.
\end{proof}

\begin{remark}
This argument yields the ``existence of small subalgebras'' result that appears as~\cite[Theorem~B]{ananyan-hochster}, and which is one the main results of that paper.  The subalgebra is $R\subseteq S$ and it is ``small'' because $s$ is independent of $n$, and thus we could have $s\ll n$.
\end{remark}

One key point that we want to emphasize is that the above argument can easily be used to bound many other important invariants or properties.  In other words, even while projective dimension was the original focus in Stillman's Conjecture, the consequences of the Ananyan--Hochster principle are much more wide-reaching.  In fact, this basic framework is so powerful that Ananyan and Hochster themselves write: ``It is difficult to make a comprehensive statement of all the related results that follow from the main theorems'' in \cite[Remark~1.4]{ananyan-hochster}.

\section{From Hilbert's Basis Theorem to $\GL$-noetherianity} \label{s:GLnoeth}

Stillman's Conjecture is a finiteness statement, as are various other instances of Stillman uniformity.  A general approach to obtaining finiteness statements in algebra is through the use of the noetherian property. In this section, we show how one can deduce Stillman uniformity from a recent noetherianity result due to Draisma.

\subsection{The classical picture of Hilbert and Noether}

We begin by recalling the definition of the noetherian property, as it appears in every graduate algebra course:

\begin{definition}
A commutative ring $R$ is {\bf noetherian} if every ascending chain $I_1 \subset I_2 \subset \cdots$ of ideals in $R$ stabilizes (i.e., satisfies $I_n=I_{n+1}$ for $n \gg 0$). Equivalently, $R$ is noetherian if every ideal of $R$ is finitely generated.
\end{definition}

The equivalence of the two conditions in the definition is a standard exercise. While it may not be apparent why the noetherian condition should be natural or important, the work of Hilbert and Noether demonstrated this convincingly. Essentially, many rings one cares about are noetherian, and the noetherian property implies most other finiteness properties of interest. The first point is a consequence of the famous Hilbert Basis Theorem:

\begin{theorem}[Hilbert's Basis Theorem]
The polynomial ring $\bC[z_1, \ldots, z_n]$ is noetherian.
\end{theorem}

As Hilbert was well aware, there is an intimate link between commutative algebra and algebraic geometry, and so Hilbert's Basis Theorem therefore has geometric implications. It is with these sorts of results that our interests lie, so we now explain them. For a set $S \subset \bC[z_1, \ldots, z_n]$ of polynomials, let $V(S) \subset \bC^n$ be their common zero locus:
\begin{displaymath}
V(S) = \{ z \in \bC^n \mid \text{$\phi(z)=0$ for all $\phi \in S$} \}.
\end{displaymath}
It is an easy exercise to verify that $V(S)=V(I)$ where $I$ is the ideal of $\bC[z_1, \ldots, z_n]$ generated by $S$, so we may as well restrict our attention to ideals when considering $V(-)$. Subsets of $\bC^n$ of the form $V(I)$ are called {\bf closed algebraic sets}, and, in the dictionary between commutative algebra and algebraic geometry, they correspond to ideals. In fact, another classical theorem of Hilbert, the Nullstellensatz, implies that $I \mapsto V(I)$ is a bijection between a certain class of ideals---the radical ideals---and closed algebraic sets. As the name suggests, there is a topology on $\bC^n$ in which the closed sets are exactly the closed algebraic sets; this is the {\bf Zariski topology}.

We have just seen that closed algebraic sets are the geometric counterpart to ideals. What then is the geometric analog of the noetherian property? Observing that an inclusion $I \subset J$ of ideals yields an inclusion $V(J) \subset V(I)$ in the opposite direction on algebraic sets, we are led to the following definition:

\begin{definition}
A topological space $X$ is {\bf noetherian} if every descending chain $\cdots \subset Z_2 \subset Z_1 \subset X$ of closed subsets of $X$ stabilizes, i.e., satisfies $Z_n=Z_{n+1}$ for $n \gg 0$.
\end{definition}

The above discussion immediately yields the following geometric form of the Hilbert Basis Theorem:

\begin{theorem}[Hilbert Basis Theorem, geometric form]
The space $\bC^n$, equipped with the Zariski topology, is a noetherian topological space.
\end{theorem}

\begin{remark}
Suppose that $V$ is a finite dimensional complex vector space. One then has the notion of a polynomial function $V \to \bC$: these are just polynomials in linear functionals on $V$. One can therefore define closed algebraic sets in $V$ and the Zariski topology on $V$, just like on $\bC^n$, and the Hilbert Basis Theorem continues to apply.
\end{remark}

\subsection{A sample application of Hilbert's Theorem} \label{ss:hilbapp}

We have stated that the noetherian property implies most other finiteness properties one might want. For seasoned algebraists, this principle is second nature. For the benefit of readers not in this group, we now provide one example.

Suppose that $f_1, \ldots, f_r \in \bC[x_1, \ldots, x_n]=R$ are homogeneous polynomials, and let $I$ be the ideal they generate. Since $I$ is a homogeneous ideal, the quotient ring $R/I$ is graded. The {\bf Hilbert function} of $R/I$ is the function $\bN \to \bN$ defined by
\begin{displaymath}
\HF_{R/I}(m) = \dim_{\bC}(R/I)_m,
\end{displaymath}
where $(R/I)_m$ denotes the degree $m$ piece of $R/I$.

\begin{example} \label{ex:hilb}
Suppose that $f \in R$ has degree $d$ and is non-zero, and consider the principal ideal $I=(f)$ it generates. The degree $m$ piece of $I$ consists of all polynomials of the form $gf$ where $g$ has degree $m-d$; note that $g$ is uniquely determined from $gf$ since $R$ is an integral domain. We thus see that the dimension of $I_m$ coincides with that of $R_{m-d}$, using the convention that this has dimension~0 for $m<d$. It follows that
\begin{displaymath}
\HF_{R/I}(m) = \dim{R_m}-\dim{R_{m-d}}.
\end{displaymath}
Since $R_m$ is the vector space of degree $m$ polynomials in $n$ variables, we have $\dim{R_m}=\binom{n+m-1}{n-1}$, and thus the above is an explicit formula. Of course, if $f=0$ then $I=(f)=0$ as well, and so $\HF_{R/I}(m)=\dim{R_m}$.
\end{example}

An interesting problem is to try to understand what the possibilities for the Hilbert function are, perhaps under constraints on the $f$'s. In general, this is a difficult problem. However, the noetherian property yields an important finiteness result, without much effort:

\begin{theorem} \label{thm:finhf}
Let $R=\bC[x_1,\dots,x_n]$ and fix $d_1, \ldots, d_r \in \bN$. As $(f_1, \ldots, f_r)$ varies over all tuples in $R$ of homogeneous polynomials of degrees $(d_1, \ldots, d_r)$ in $n$ variables, only finitely many Hilbert functions appear.
\end{theorem}

\begin{proof}
Let $X_{d,n}$ be the set of all homogeneous polynomials of degree $d$ in $n$ variables, and let $Y=X_{d_1,n} \times \cdots \times X_{d_r,n}$. This is a finite dimensional complex vector space, and thus carries a Zariski topology; furthermore, equipped with this topology, $Y$ is a noetherian space by Hilbert's Basis Theorem.  Each point $y\in Y$ corresponds to a tuple $(f_1,\dots,f_r)$ of polynomials in $\bC[x_1,\dots,x_n]$ with $\deg(f_i)=d_i$.  For a point $y \in Y$, we let $H_y$ denote the Hilbert function for $R/I$ where $I$ is generated by the tuple corresponding to $y$.  We must show that the set $\{H_y \mid y\in Y\}$ is finite.

A common theme in algebraic geometry is that objects exhibit generic behavior. We now explain what this means for $H_y$. Suppose that $Z \subset Y$ is a non-empty Zariski closed set. By noetherianity of $Y$, the space $Z$ can be written as a finite union $Z_1 \cup \cdots \cup Z_k$ where each $Z_i$ is an irreducible closed set (i.e., it does not non-trivially decompose into a union of closed subsets). One can then show, using standard algebraic methods, that each $Z_i$ contains a non-empty open subset $U_i$ such that $y \mapsto H_y$ is constant on $U_i$. This is what we mean by $H$ admitting a generic behavior. We saw this already in Example~\ref{ex:hilb}: there the generic behavior occurred on the open set $f \ne 0$, while degenerate behavior appeared on the closed set $f=0$.

Keeping the above notation, let $U=U_1 \cup \cdots \cup U_k$. This is a non-empty open subset of $Z$, and so its complement $Z'=Z \setminus U$ is a proper closed subset of $Z$. Furthermore, we know that off of $Z'$, we see only finitely many values for $H$, since $H$ is constant on each $U_i$.

The noetherian property now gives us the desired result. Indeed, let $Z_1=Y$. The previous paragraph produces a proper closed subset $Z_2 \subset Z_1$ (what was called $Z'$ there) such that $H$ takes on finitely many values on $Z_1 \setminus Z_2$. Now apply the previous paragraph again to $Z_2$, assuming it is non-empty, and get $Z_3 \subset Z_2$ with analogous behavior. This process thus produces a strictly descending chain $\cdots \subset Z_2 \subset Z_1$ of closed subsets of $Y$, and therefore must terminate in finitely many steps by the noetherian property. Since $H$ takes finitely many values on each piece $Z_i \setminus Z_{i+1}$ and there are finitely many pieces, the result follows.
\end{proof}

\subsection{Draisma's Theorem}

In the previous section, we considered the set $X_{d,n}$ of homogeneous polynomials of degree $d$ in $n$ variables as a geometric object, and saw that we could use ideas from algebraic geometry to prove an interesting result about polynomials. We would now like to apply similar ideas to the study of polynomials of fixed degree in an arbitrary number of variables to prove instances of Stillman uniformity. As a first step in this direction, we need an analog of Hilbert's Basis Theorem in this new setting.

To begin, let $X_d$ denote the space of all homogeneous polynomials of degree $d$ in an arbitrary number of variables. Like $X_{d,n}$, this is a complex vector space; however, unlike $X_{d,n}$, it is infinite dimensional. Nonetheless, we define the Zariski topology on $X_d$ analogously to before. Precisely, an element of $X_d$ can be written in the form $\sum c_{\alpha} x^{\alpha}$, the sum being over degree $d$ multi-indices $\alpha$. One can regard the $c_{\alpha}$'s as coordinate functions on $X_d$. By a polynomial function on $X_d$, we mean a polynomial in the $c_{\alpha}$'s. A closed algebraic set in $X_d$ is then a set that can be realized as the common zero locus of a set of polynomial functions.\footnote{The astute reader will recognize $X_d=\varinjlim X_{d,n}$ as an ind-variety. The ``Zariski topology'' defined above is ad hoc, and does not come from a general construction on ind-varieties. However, for $\GL_{\infty}$-stable sets, the condition of being closed in our topology is the same as being closed in the ind-topology.  See \cite[\S2]{genstillman}.} More generally, one can define closed algebraic subsets of $X_{d_1} \times \cdots \times X_{d_r}$ for any $d_1, \ldots, d_r \in \bN$.

Based on the discussion thus far, one might expect us to now say that $X_d$ is a noetherian topological space. However, this is far from the truth: it is simply too large! For example, just consider the case $d=1$. An element of $X_1$ can be written in the form $\sum_{i \ge 1} c_i x_i$. Let $Z_n \subset X_1$ be the closed algebraic set defined by $c_1=c_2=\cdots=c_n=0$. Then the $Z$'s form an infinite strictly descending chain of closed sets, which shows that $X_1$ is not noetherian. Similar examples can be constructed for $X_d$, for any $d \ge 1$.

Not long ago, this would have been the anti-climactic end of the story. However, in the last decade an important principle has emerged (which is the basis of the burgeoning field of representation stability, as well as the results discussed in \S\ref{subsec:GL}): many large objects that have a large amount of symmetry are noetherian when the symmetry is appropriately taken into account. For the present situation, the following definition makes this idea precise:

\begin{definition}
Let $X$ be a topological space equipped with an action of a group $G$. We say that $X$ is {\bf $G$-noetherian} if every descending chain of $G$-stable closed subsets stabilizes.
\end{definition}

To consider this property in relation to the space $X_d$, we must first specify our group. The group $\GL_n$ of invertible $n \times n$ matrices acts on the set $X_{d,n}$ of homogeneous polynomials of degree $d$ in $n$ variables via linear changes of variables. Thus the group $\GL_{\infty}=\bigcup_{n \ge 1} \GL_n$ acts on the set $X_d=\bigcup_{n \ge 1} X_{d,n}$. Since most natural properties in commutative algebra (such as projective dimension of an ideal) are invariant under linear changes of variables, we can expect most interesting subsets of $X_d$ to be $\GL_{\infty}$-stable. It is therefore reasonable to restrict our attention to $\GL_{\infty}$-stable subsets.

\begin{example}
Any two non-zero elements of $X_1$ belong to the same $\GL_{\infty}$-orbit. Thus there are precisely three $\GL_{\infty}$-stable closed subsets of $X_1$: the whole space, the set $\{0\}$, and the empty set.  It follows that $X_1$ is $\GL_\infty$-noetherian.
\end{example}

\begin{example}
An element of $X_2$ can be regarded as a quadratic form in some number of variables, and thus has an associated rank (the rank of the corresponding symmetric matrix). The theory of quadratic forms over the complex numbers shows that any two forms of the same rank belong to the same $\GL_{\infty}$-orbit. Furthermore, it is not difficult to see that the set $Z_n \subset X_2$ of forms of rank $\le n$ is closed: it is the common zero locus of certain minors of the corresponding symmetric matrix. We thus see that the $\GL_{\infty}$-stable closed subsets of $X_2$ are the $Z_n$'s together with the whole space and the empty space. It follows that $X_2$ is $\GL_{\infty}$-noetherian.
\end{example}

As we have just seen, one can classify the $\GL_{\infty}$-stable closed subsets of $X_1$ and $X_2$, and thus establish the $\GL_{\infty}$-noetherian property for them ``by hand.'' However, in most other situations, the $\GL_{\infty}$-stable closed subsets defy classification, and proving the property is non-trivial. The first major step forward came in Eggermont's work \cite{eggermont}, in which he established the noetherian property for $X_2 \times \cdots \times X_2$ (any number of factors). Shortly thereafter, Derksen, Eggermont, and Snowden \cite{DES} handled the $X_3$ case. And shortly after that, Draisma \cite{draisma} established the general case\footnote{In fact, all the cited works prove their results for the \emph{inverse} limit of the $X_{d,n}$, while our $X_d$ is the \emph{direct} limit. However, in \cite{genstillman} we showed that one can go back and forth between the two spaces.}:

\begin{theorem}[Draisma]
For any $d_1,\ldots,d_r \in \bN$, the space $X_{d_1} \times \cdots \times X_{d_r}$ is $\GL_{\infty}$-noetherian.
\end{theorem}

We have thus completed our goal of finding an analog of the Hilbert Basis Theorem in the setting of Stillman uniformity.  (In fact, Draisma proves an even stronger result that applies to more general polynomial functors, but a discussion of this would take us too far afield.)

\subsection{Application of Draisma's Theorem to Stillman uniformity}

In \S \ref{ss:hilbapp}, we saw that Hilbert's Basis Theorem, in its geometric form, could be applied to obtain interesting finiteness properties of polynomials in a fixed number of variables. We now show that, in exactly the same way, Draisma's Theorem can be used to obtain finiteness properties for polynomials in an arbitrary number of variables. Specifically, we present the second proof of Stillman's Conjecture from \cite{stillman}.

\begin{theorem}[Stillman's Conjecture] \label{thm:stillmanR}
Fix $d_1,\ldots,d_r \in \bN$. Then there exists $B \in \bN$ with the following property: if $f_1, \ldots, f_r$ are homogeneous polynomials of degrees $(d_1, \ldots, d_r)$ then the ideal they generate has projective dimension $\le B$.
\end{theorem}

\begin{proof}
Let $Y=X_{d_1} \times \cdots \times X_{d_r}$. Each point of $y \in Y$ corresponds to some tuple $(f_1,\ldots,f_r)$ and thus has an associated ideal and projective dimension $p(y)$. We aim to show that $\{ p(y) \mid y \in Y\}$ is finite following the argument used in the proof of Theorem~\ref{thm:finhf}.

To apply this argument, we need to know that $p$ exhibits generic behavior. In fact, this is true: given any irreducible closed subset $Z$ of $Y$, there is a non-empty open subset of $Z$ on which $p$ is constant. This follows from \cite[Theorem~5.12]{stillman}. In the setting of Theorem~\ref{thm:finhf}, we said that the generic behavior of Hilbert functions can be established by standard methods in algebraic geometry. To get the generic behavior of $p$, we employ the same methods; however, since we are now in an infinite dimensional setting, there are a number of complications, and our proof crucially relies on the idealized Ananyan--Hochster principle for $\bR$ (i.e., Corollary~\ref{cor:inv limit}) to ensure that things behave well.

The rest of the argument now goes through identically. Start with $Z_1=Y$. There is then an open set $Z_1 \setminus Z_2$ on which $p$ takes finitely many values. Now argue similarly for $Z_2$, and then use Draisma's Theorem to conclude that the chain $\cdots \subset Z_2 \subset Z_1$ is finite. There are two facts we have tacitly used here: first, a $\GL_{\infty}$-stable subset of $Y$ has finitely many irreducible components, each of which is $\GL_{\infty}$-stable; and second, one can take the set $Z_2$ (and all subsequent $Z_i$'s) to be $\GL_{\infty}$-stable, essentially because $p$ is $\GL_{\infty}$-invariant (that is, $p(gy)=p(y)$ for $y \in Y$ and $g \in \GL_{\infty})$.
\end{proof}

\begin{remark}
The proof of Stillman's Conjecture given above appears in~\cite[\S5]{stillman}.  It provides a third proof, following Ananyan--Hochster's original proof, and our ultraproduct proof from \cite{stillman} outlined in \S \ref{s:reggerm}--\S \ref{s:proof of stillman}. This third proof is totally different in character from those other two proofs, as it does not go through the theory of small subalgebras.
\end{remark}

\begin{remark}
Draisma's Theorem can be used to establish many other instances of Stillman uniformity; see \S \ref{ss:genstillman}.
\end{remark}

\subsection{Draisma--Laso\'n--Leykin's initial ideal proof}
One seemingly natural way to approach Stillman's Conjecture would be through the theory Gr\"obner bases.  Shortly after \cite{stillman} appeared, Draisma, Laso\'n, and Leykin \cite{draisma-lason-leykin} used this type of approach to give a fourth proof of Stillman's Conjecture.  We now give a very brief overview of this fourth proof; we assume familiarity with topics related to Gr\"obner bases, such as generic initial ideals.

The essential idea in \cite{draisma-lason-leykin} is to develop a good theory of Gr\"obner bases in the ring $\bR$ of bounded-degree series, using the graded revlex term order.\footnote{The graded revlex term order is a natural term order for $\bR$ because it interacts well with the maps $\bR\to \bC[x_1,\dots,x_n]$ for all $n$.}  In particular, Draisma, Laso\'n, and Leykin develop a version of Buchberger's Algorithm which works for bounded-degree series.  The finiteness properties of this algorithm stem from Draisma's Theorem and from Corollary~\ref{cor:inv limit}, and yield a finiteness result for generic initial ideals of ideals generated in specified degrees~\cite[Theorem~3]{draisma-lason-leykin}.  Since projective dimension is invariant under passing to the generic initial ideal (using graded revlex)~\cite{bayer-stillman}, this implies Stillman's Conjecture.

\section{Related Topics}\label{s:topics}
We end by discussing some topics related to the questions raised in Stillman's Conjecture and elsewhere.

\subsection{Changing the base field}\label{subsec:field}
In this article, we chose the base field $\bC$ for expository purposes.  Just about every theorem in the paper holds over an arbitrary field $\bk$.
(Due to the use of partial derivatives, the proof of Theorem~\ref{thm:algindseries} outlined in \S\ref{subsec:proof 4} requires much more care in positive characteristic.  But the central ideas remain the same.)

Ananyan and Hochster's original work was over an arbitrary algebraically closed field~\cite{ananyan-hochster}; this is sufficient for proving Stillman's Conjecture over any field, though some of their auxiliary results on strength and its consequences were not known over other fields.  We extended many---but not all---of those results to an arbitrary perfect field in~\cite{stillman} and then to possibly imperfect fields in~\cite{imperfect}.  The initial ideal proof of~\cite{draisma-lason-leykin} also holds over an arbitrary field.  

While one can easily define $\nu$-complexity over any field, some incarnations of the Ananyan--Hochster principle that hold over perfect fields will fail over fields that are not perfect.  This is because $\nu$-complexity can depend on the base field (i.e., it can change after passing to a larger field). For instance, let $\bk=\bF_p(a_1,a_2,a_3,\dots)$ be a field where the $a_i$ are independent transcendental elements.  For each $n$, the polynomial $f_n=\sum_{i=1}^n a_ix_i^p$ has $\nu$-complexity $n$.  However, if we pass to the perfect closure of $\bk$, then $f_n$ factors as $f_n=(\sum_{i=1}^n \sqrt[p]{a_i}x_i)^p$, and it thus has $\nu$-complexity $1$ over this field extension.  For an imperfect field $\bk$, it would be interesting to better understand exactly which implications of the Ananyan--Hochster principle depend only on the value of $\nu$ over that field $\bk$.

\subsection{Effective Bounds}
In their first paper on Stillman's Conjecture, which focused on the case where $f_1,\dots,f_r$ were all quadratic polynomials, Ananyan and Hochster gave an asymptotic bound of $(2r)^{2r}$ on the projective dimension of the ideal generated by $f_1,\dots,f_r$~\cite[\S6]{ananyan-hochster-quadrics}.
For the general case, Ananyan and Hochster's results in~\cite{ananyan-hochster} do lead to an effective bound for the function $B(d_1,\dots,d_r)$ from Stillman's Conjecture (see Conjecture~\ref{conj:Stillman} above), but the bound would involve iterated exponential functions, and has not yet been written out explicitly.   
By contrast, the methods of~\cite{stillman} and~\cite{draisma-lason-leykin} are inherently ineffective.

In a different direction, there has been some work on producing families of examples that provide lower bounds for $B(d_1,\dots,d_r)$.  For instance, if $f_1,\dots,f_r$ have degree $d$, then ~\cite{mccullough} produces a family of ideals whose projective dimension grows like $d^{r-2}$ as $d\to \infty$.  See also Theorem~7.6 above, which comes from~\cite{bmnsss}.

Finally, there is work on producing tight bounds in special cases like 3 cubics or 4 quadrics; see~\cite{engheta1, engheta2, hmms, mantero-mccullough}.  There are a great many open questions in this area and the expository article~\cite{ms} provides a nice introduction (though it was written before many of these recent advances).

\subsection{Variants of Stillman's Conjecture with different inputs}
One can ask whether analogues of Stillman's Conjecture hold where, instead of fixing the degrees of the forms, one instead fixes some other invariants of the ideal.  For instance, ~\cite{cmpv} shows that the projective dimension a nondegenerate prime ideal can be bounded by a function of its degree.

There are also negative results in this vein.  The results of ~\cite{hmms-2} produce primary ideals of bounded multiplicity and codimension, but with arbitrarily large projective dimension.  In~\cite{mccullough-exterior}, McCullough shows that for an ideal in an exterior algebra, there is no Stillman-type bound on Castelnuovo--Mumford regularity (the projective dimension of ideals in the exterior algebra is typically infinite, but regularity is finite).

\subsection{Variants of Stillman's Conjecture with different outputs} \label{ss:genstillman}
In~\cite{genstillman}, we consider generalizations of Stillman's Conjecture where the input is the same, but where we bound invariants other than projective dimension.  

We define an {\bf ideal invariant} as a rule $\tau$ that associates to each homogeneous ideal $I\subseteq \bk[x_1, \ldots, x_n]$ a quantity $\tau(I) \in \bZ \cup \{\infty\}$, and where $\tau(I)$ that is invariant under linear changes of coordinates of the polynomial ring.
We say $\tau$ is {\bf degreewise bounded} if there exists a function $B(d,r)$ such that $\tau(I)\leq B(d,r)$ or $\tau(I)=\infty$ for every ideal $I$ which is generated by $r$ polynomials of degree at most $d$; crucially, $B(d,r)$ does not depend on the number of variables.  In this language, Stillman's Conjecture says that the invariant ``projective dimension'' is degreewise bounded.

To obtain boundedness results, we require two niceness conditions on our invariants.  First, we say that $\tau$ is {\bf cone-stable} if adjoining a new variable does not affect its value.  Second, we say that $\tau$ is {\bf weakly upper semi-continuous} if it is upper semi-continuous in any flat family of ideals.  
(Many interesting invariants, including projective dimension and Castelnuovo--Mumford regularity, are weakly upper semi-continuous but not upper  semi-continuous.) In~\cite[Theorem~1.1]{genstillman} we prove:

\begin{theorem}
Any ideal invariant that is cone-stable and weakly upper semi-continuous is degreewise bounded.
\end{theorem}

This provides many new variants of the Stillman uniformity phenomena described in the introduction.  It is also closely connected to $\GL$-noetherianity results like Draisma's Theorem and those discussed below.

\subsection{Hartshorne's Conjecture}
Hartshorne famously conjectured that every smooth subvariety $X\subseteq \bP^n$ of codimension $c$ must be a complete intersection if $c < \frac{1}{3}n$~\cite{hartshorne-bulletin}.   In~\cite{hartshorne}, we used the circle of ideas related to the Ananyan--Hochster principle to give a simple proof of a special case of this conjecture.  In particular, we showed that if one fixes $c$ and the degree of $X$, then Hartshorne's Conjecture holds whenever $n\gg c,\deg(X)$.  This extends results of Hartshorne, Barth--Van de Ven, and many other authors~\cite{ballico-chiantini, barth-icm, BVdV, BVdV-Grassman, bertram-ein-lazarsfeld, ran} from characteristic zero to arbitrary characteristic.

\subsection{More on $\GL$-noetherianity} \label{subsec:GL}
Draisma's Theorem is closely connected
to some specific conjectures of the third author that arose in his work on syzygies of Segre embeddings~\cite{snowden-segre}, which propose that twisted commutative algebras might satisfy certain noetherianity conditions.  Similar ideas were being developed in the work of Church, Ellenberg, and Farb on FI-modules, and which also revealed noetherianity conditions that held up to a certain group action~\cite{cef}.   In addition, the special cases of Draisma's Theorem for quadratic polynomials was shown by Eggermont~\cite{eggermont} and the case of a single cubic polynomial was proven by Derksen, Eggermont, and Snowden~\cite{DES}.

Finally, recall that topological noetherianity of $\bC^n$ follows from the Hilbert Basis Theorem applied to $\bC[x_1,\dots,x_n]$, which is a much stronger statement. One can conjecture $\GL$-analogues of the Hilbert Basis Theorem from which the above theorems would follow. Some work in this direction can be found in \cite{sym2noeth}, though there are still many open questions in this vein.  See also~\cite{draisma-survey}.

\subsection{Universality of strength}\label{subsec:bde}
A recent result of Bik, Draisma, Eggermont further underscores the centrality of the notion of strength.  (Recall that strength is defined in Remark~\ref{rmk:strength}, and that it is asymptotically equivalent to $\nu$-complexity.)  They prove that for any $\GL(W)$-invariant Zariski closed condition, polynomials of high enough strength will not satisfy that closed condition~\cite{bik-draisma-eggermont}.  This provides another way to make the Ananyan--Hochster principle precise.  It also generalizes a theorem of Kazhdan and Ziegler which bounds the strength of a polynomial in terms of the strength of its partial derivatives~\cite{kazhdan-ziegler}.

\end{document}